\pdfoutput=1
\RequirePackage{ifpdf}
\ifpdf 
\documentclass[pdftex]{sigma}
\else
\documentclass{sigma}
\fi

\numberwithin{equation}{section}

\newtheorem{Theorem}{Theorem}[section]
\newtheorem{Corollary}[Theorem]{Corollary}
\newtheorem{Lemma}[Theorem]{Lemma}
\newtheorem{Proposition}[Theorem]{Proposition}
 { \theoremstyle{definition}
\newtheorem{Definition}[Theorem]{Definition}
\newtheorem{Remark}[Theorem]{Remark}
\newtheorem{Example}[Theorem]{Example}
 }

\newcommand{\Set}[1]{\left\{#1\right\}}
\newcommand{\setDef}[2]{{#1}\left|\,\vphantom{#1}{#2}\right.}
\newcommand{\SetDef}[2]{\Set{\setDef{#1}{#2}}}

\begin{document}

\newcommand{\arXivNumber}{2007.09950}

\renewcommand{\thefootnote}{}

\renewcommand{\PaperNumber}{019}

\FirstPageHeading

\ShortArticleName{Computing Regular Meromorphic Differential Forms via Saito's Logarithmic Residues}

\ArticleName{Computing Regular Meromorphic Differential Forms\\ via Saito's Logarithmic Residues\footnote{This paper is a~contribution to the Special Issue on Primitive Forms and Related Topics in honor of~Kyoji Saito for his 77th birthday. The full collection is available at \href{https://www.emis.de/journals/SIGMA/Saito.html}{https://www.emis.de/journals/SIGMA/Saito.html}}}

\Author{Shinichi TAJIMA~$^{\rm a}$ and Katsusuke NABESHIMA~$^{\rm b}$}

\AuthorNameForHeading{S.~Tajima and K.~Nabeshima}

\Address{$^{\rm a)}$~Graduate School of Science and Technology, Niigata University, \\
\hphantom{$^{\rm a)}$}~8050, Ikarashi 2-no-cho, Nishi-ku Niigata, Japan }
\EmailD{\href{mailto:tajima@emeritus.niigata-u.ac.jp}{tajima@emeritus.niigata-u.ac.jp}}

\Address{$^{\rm b)}$~Graduate School of Technology, Industrial and Social Sciences, Tokushima University, \\
\hphantom{$^{\rm b)}$}~2-1, Minamijosanjima-cho, Tokushima, Japan}
\EmailD{\href{mailto:nabeshima@tokushima-u.ac.jp}{nabeshima@tokushima-u.ac.jp}}

\ArticleDates{Received July 24, 2020, in final form February 05, 2021; Published online February 27, 2021}

\Abstract{Logarithmic differential forms and logarithmic vector fields associated to a hypersurface with an isolated singularity are considered in the context of computational complex analysis. As applications, based on the concept of torsion differential forms due to A.G.~Aleksandrov, regular meromorphic differential forms introduced by D.~Barlet and M.~Kersken, and Brieskorn formulae on Gauss--Manin connections are investigated. $(i)$~A~method is given to describe singular parts of regular meromorphic differential forms in terms of non-trivial logarithmic vector fields via Saito's logarithmic residues. The resulting algorithm is illustrated by using examples. $(ii)$~A new link between Brieskorn formulae and logarithmic vector fields is discovered and an expression that rewrites Brieskorn formulae in terms of non-trivial logarithmic vector fields is presented. A new effective method is described to compute non trivial logarithmic vector fields which are suitable for the computation of~Gauss--Manin connections. Some examples are given for illustration.}


\Classification{32S05; 32A27}

\begin{flushright}
\begin{minipage}{60mm}
{\it Dedicated to Kyoji Saito } \\
{\it on the occasion of his $77^{th}$ birthday}
\end{minipage}
\end{flushright}

\renewcommand{\thefootnote}{\arabic{footnote}}
\setcounter{footnote}{0}

\section{Introduction}
In 1975, K.~Saito introduced, with deep insight, the concept of logarithmic differential forms and that of logarithmic vector fields and studied Gauss--Manin connection associated with the versal deformations of hypersurface singularities of type $A_2$ and $A_3$ as applications. These results were published in~\cite{S77}. He developed the theory of logarithmic differential forms, logarithmic vector fields and the theory of residues and published in 1980 a landmark paper~\cite{S}. One of the motivations of his study, as he himself wrote in~\cite{S}, came from the study of Gauss--Manin connections~\cite{B,S73}. Another motivation came from the importance of these concepts he realized.
Notably the logarithmic residue, interpreted as a meromorphic differential form on~a~divisor, is~regarded as a natural generalization of the classical Poincar\'e residue to the singular cases.

In 1990, A.G.~Aleksandrov~\cite{A} studied Saito theory and gave in particular a characterization of the image of the residue map. He showed
that the image sheaf of the logarithmic residues coincides with the sheaf of regular meromorphic differential forms introduced by D.~Barlet ~\cite{B} and M.~Kersken~\cite{K83,K84}.
We refer the reader to~\cite{AT,Bru,CM1,CM2,GS,P,P2} for more recent results on logarithmic residues.

We consider logarithmic differential forms along a hypersurface with an isolated singularity in the context of computational complex analysis. In our previous paper~\cite{TN20}, we study torsion modules and give an effective method for computing them.
In the present paper, we first consider a method for computing regular meromorphic differential forms. We show that, based on the result of A.G.~Aleksandrov mentioned above, representatives of regular meromorphic differential forms can be computed by adapting the method presented in~\cite{TN20} on torsion modules. Main ideas of our approach are the use of the concept of logarithmic residues and that of logarithmic vector fields. Next, we discuss a relation between logarithmic differential forms and Brieskorn formulae~\cite{B,Sch,Schu} and we show that Brieskorn formulae can be rewritten in~terms of logarithmic vector fields. Applications to the computation of Gauss--Manin connections are illustrated by using examples.

In Section~\ref{sec2}, we briefly recall some basics on logarithmic differential forms,
logarithmic residues, Barlet sheaf and torsion differential forms.
In Section~\ref{sec3}, we first recall the notion of logarithmic vector fields and a result
gave in~\cite{TN20} to show that torsion differential forms can be described
in terms of non trivial logarithmic vector fields. Next, we recall our previous
results to show that non-trivial logarithmic vector fields can be computed
by using a polar method and local cohomology. Lastly in Section~\ref{sec3}, we present
Theorem~\ref{Th6} which say that regular meromorphic differential forms can be
explicitly computed by modifying our previous algorithm on torsion differential
forms.
In Section~\ref{sec4}, we give some examples to illustrate the proposed method of computing
non-trivial logarithmic vector fields and regular meromorphic differential forms.
In Section~\ref{sec5}, we consider Brieskorn formulae on Gauss--Manin connections. We show
that Brieskorn formulae described in terms of logarithmic differential forms can be
rewritten in terms of non-trivial logarithmic vector fields. We give a new method for computing
non-trivial logarithmic vector fields which is suitable in use to compute a connection
matrix of Gauss--Manin connections. Finally, we show that the use of integral dependence relations
provides a~new effective tool for computing saturations of Gauss--Manin connection.

\section{Logarithmic differential forms and residues}\label{sec2}

In this section, we briefly recall the concept of logarithmic differential forms and that of loga\-rithmic residues and fix notation. We refer the reader to~\cite{S} for details. Next we recall the result of A.G.~Aleksandrov on regular meromorphic differential forms. Then, we recall a result of
G.-M.~Greuel on torsion modules.

Let $X$ be an open neighborhood of the origin $O$ in $ {\mathbb C}^n$. Let $ {\mathcal O}_X $ be the sheaf on $ X $ of~holo\-morphic functions and $ {\mathcal O}_{X,O} $ the stalk at $ O $ of the sheaf $ {\mathcal O}_X$.

\subsection{Logarithmic residues}

Let $f$ be a holomorphic function defined on $X$. Let $S=\{ x \in X \,|\, f(x)=0 \} $ denote the hypersurface defined by~$ f$.

\begin{Definition}\label{def1}
Let $\omega $ be a meromorphic differential $q$-form on $X$, which may have poles only along $S$. The form $\omega$ is a logarithmic differential form along $S$ if it satisfies the following equi\-valent four conditions:
\begin{enumerate}\itemsep=0pt
\item[$(i)$] $f\omega$ and $f{\rm d}\omega$ are holomorphic on $X$.
\item[$(ii)$] $f\omega$ and ${\rm d}f \wedge \omega$ are holomorphic on $X$.
\item[$(iii)$] There exist a holomorphic function $g(x)$ and a holomorphic $(q-1)$-form $\xi$ and a holomorphic $q$-form $\eta$ on~$X$,
such that:
\begin{enumerate}\itemsep=0pt
\item[$(a)$] $\dim_{\mathbb C}( S \cap \{x \in X \,|\, g(x)=0 \}) \leq n-2, $
\item[$(b)$] ${\displaystyle g\omega = \frac{{\rm d}f}{f} \wedge \xi + \eta.}$
\end{enumerate}
\item[$(iv)$] There exists an $ (n-2)$-dimensional analytic set $A \subset S$ such that the germ of $\omega$ at any point $ p \in S-A$ belongs to $ {\frac{{\rm d}f}{f} \wedge \Omega_{X, p}^{q-1} + \Omega_{X,p}^{q}},$ where $ \Omega_{X,p}^{q}$ denotes the module of germs of~holomorphic $q$-forms on~$X$ at~$p$.
\end{enumerate}
\end{Definition}

For the equivalence of the condition above, see~\cite{S}.
Let $\Omega_X^{q}(\log S)$ denote the sheaf of~loga\-rith\-mic $q$-forms along $S$.
Let ${\mathcal M}_{S}$ be the sheaf on $S$ of meromorphic functions, let $\Omega_{S}^{q}$ be the sheaf on $S$ of holomorphic $q$-forms defined to be
\begin{gather*}
\Omega_{S}^{q} = \Omega_{X}^{q}/\big(f\Omega_X^{q} + {\rm d}f\wedge \Omega_{X}^{q-1}\big).
\end{gather*}

\begin{Definition}
The residue map ${\rm res}\colon \Omega_X^{q}(\log S) \longrightarrow {\mathcal M}_S \otimes_{{\mathcal O}_X}\Omega_S^{q-1} $ is defined as follows: For~$\omega \in \Omega_S^{q}(\log S)$, by definition,
there exist $ g$, $\xi$ and $\eta$ such that
\begin{enumerate}\itemsep=0pt
\item[$(a)$] $\dim_{\mathbb C}( S \cap \{x \in X \,|\, g(x)=0 \}) \leq n-2$, and
\item[$(b)$] ${\displaystyle g\omega = \frac{{\rm d}f}{f} \wedge \xi + \eta}$.
\end{enumerate}
Then the residue of $\omega$ is defined to be
$ {\rm res}(\omega) = \frac{\xi}{g}\Big|_S$ in $ {\mathcal M}_S \otimes_{{\mathcal O}_X}\Omega_S^{q-1}$.
\end{Definition}

Note that it is easy to see that the image sheaf of the residue map $ {\rm res} $ of the subsheaf ${\frac{{\rm d}f}{f}\wedge \Omega_X^{q-1} + \Omega_X^{q}} $ of $ \Omega_X^{q}(\log S)$ is equal to $\left. \Omega_X^{q-1}\middle|\raisebox{-0.7ex}[1ex][-2ex]{{\hspace{-1mm}\;$_S$}}\right.$:
\begin{gather*}
{\rm res}\bigg( \frac{{\rm d}f}{f}\wedge \Omega_X^{q-1} + \Omega_X^{q}\bigg) = \left.\Omega_X^{q-1}\middle|\raisebox{-0.7ex}[1ex][-2ex]{{\hspace{-1mm}\;$_S$}}\right..
\end{gather*}

See also~\cite{S} for details on logarithmic residues.
The concept of residues for logarithmic differential forms can be actually regarded as a natural generalization of the classical Poincar\'e residue.

\subsection{Barlet sheaf and torsion differential forms}

In 1978, by using results of F.~El~Zein on fundamental classes, D.~Barlet introduced in~\cite{B} the notion of the sheaf $\omega_S^{q}$ of regular meromorphic differential forms in a quite general setting. He~showed that for the case $q=n-1$, the sheaf $ \omega_S^{n-1}$ coincides with the
Grothendieck dualizing sheaf and $ \omega_S^{q} $ can also be defined in the following manner.

\begin{Definition}
 Let $S$ be a hypersurface in $X \subset {\mathbb C}^n$. Let $ \omega_S^{n-1} $ be the Grothendieck dua\-li\-zing sheaf $ {\rm Ext}_{{\mathcal O}_X}^{1}\big({\mathcal O}_S, \Omega_X^{n}\big)$. Then, the sheaf of regular meromorphic differential forms $ \omega_S^{q}$, $q=0,1,\dots, n-2$ on $S$ is defined to be
\begin{gather*}
 \omega_S^{q} = {\rm Hom}_{{\mathcal O}_S}\big(\Omega_S^{n-1-q}, \omega_S^{n-1}\big).
\end{gather*}

\end{Definition}

In 1990, A.G.~Aleksandrov~\cite{A} obtained the following result.

{\samepage\begin{Theorem}
 For any $ q \geq 0$, there is an isomorphism of $ {\mathcal O}_S $ modules
\begin{gather*}
{\rm res}\big(\Omega_X^{q}(\log S)\big) \cong \omega_S^{q-1}.
\end{gather*}

\end{Theorem}

See~\cite{A} or~\cite{A05} for the proof.

}

Let $ {\rm Tor}(\Omega_S^{q}) $ denote the sheaf of torsion differential $q$-forms of $ \Omega_S^{q}$.

\begin{Example}\label{Ex1}
Let $ X $ be an open neighborhood of the origin $O$ in $ \mathbb{C}^2.$ Let $ f(x,y)=x^2-y^3$ and $ S=\{ (x,y) \in X \,|\, f(x,y)=0 \}$. Then, for stalk at the origin of the sheaves of logarithmic differential forms, we have
\begin{gather*}
\Omega_{X,O}^{1}(\log S) \cong {\mathcal O}_{X,O}\bigg(\frac{{\rm d}f}{f}, \frac{\beta}{f}\bigg), \qquad \Omega_{X,O}^{2}(\log S) \cong {\mathcal O}_{X,O}\bigg(\frac{{\rm d}x\wedge {\rm d}y}{f}\bigg),
\end{gather*}
where ${\mathcal O}_{X,O}$ is the stalk at the origin of the sheaf ${\mathcal O}_X$ of holomorphic functions and $\beta=2y {\rm d}x-3x{\rm d}y$. The~dif\-ferential form $\beta$, as an element of $\Omega_S^{1} = \Omega_X^{1}/\big(\mathcal{O}_{X}{\rm d}f +f\Omega_X^{1}\big)$, is a torsion. The~dif\-ferential form $y\beta$ is also a torsion. Since the defining function $f$ is quasi-homogeneous, the dimension of the vector space ${\rm Tor}\big(\Omega_S^{1}\big)$ is equal to the Milnor number $\mu=2$ of~$S$~\cite{M, Z}. Therefore we have
$ {\rm Tor}\big(\Omega_S^{1}\big) \cong {\mathcal O}_{X,O}(\beta) \cong {\mathbb C}(\beta, y\beta)$.
\end{Example}

In 1988~\cite{A88}, A.G.~Aleksandrov studied logarithmic differential forms and residues and proved in particular the following.

\begin{Theorem}
Let $S=\{ x \in X \,|\, f(x) =0 \} $ be a hypersurface in $X \subset {\mathbb C}^n$. For $ q=0,1,\ldots,n $, there exists an exact sequence of sheaves of $ {\mathcal O}_{X} $ modules,
\begin{gather*}
0 \longrightarrow\frac{{\rm d}f}{f} \wedge \Omega_{X}^{q-1}+ \Omega_{X}^{q} \longrightarrow
\Omega_{X}^{q}(\log S) \stackrel{\cdot f}\longrightarrow {\rm Tor}\big(\Omega_{S}^{q}\big) \longrightarrow 0.
\end{gather*}
\end{Theorem}

The result above yields the following observation:
$ {\rm Tor}\big(\Omega_{S}^{q}\big)$ plays a key role to study the structure of $ {\rm res}\big(\Omega_X^{q}(\log S)\big)$.

\subsection{Vanishing theorem}

In 1975, in his study~\cite{G} on Gauss--Manin connections G.-M.~Greuel proved the following results on torsion differential forms.

\begin{Theorem}
Let $ S = \{ x \in X \,|\, f(x) = 0 \} $ be a hypersurface in $X$
with an isolated singularity at $ O \in {\mathbb C}^n$.
Then,
\begin{enumerate}\itemsep=0pt
\item[$(i)$] $ {\rm Tor}\big(\Omega_S^{q}\big) = 0$, $q=0,1,\ldots,n-2$.

\item[$(ii)$] ${\rm Tor}\big(\Omega_S^{n-1}\big)$ is a skyscraper sheaf supported at the origin~$O$.

\item[$(iii)$] The dimension, as a vector space over ${\mathbb C}$, of the torsion module ${\rm Tor}\big(\Omega_S^{n-1}\big) $ is equal to~$\tau(f)$, the Tjurina number of the hypersurface $S$ at the origin defined to be
\begin{gather*}
 \tau(f) = \dim_{{\mathbb C}}\bigg({\mathcal O}_{X, O}\Big/\bigg(f, \frac{\partial f}{\partial x_1}, \frac{\partial f}{\partial x_2},\ldots,\frac{\partial f}{\partial x_n}\bigg)\bigg),
\end{gather*}
where
$ \big(f, \frac{\partial f}{\partial x_1}, \frac{\partial f}{\partial x_2},\ldots,
\frac{\partial f}{\partial x_n}\big) $
is the ideal in $ {\mathcal O}_{X,O} $ generated by
$f, \frac{\partial f}{\partial x_1}, \frac{\partial f}{\partial x_2}, \ldots, \frac{\partial f}{\partial x_n}$.
\end{enumerate}
\end{Theorem}

Note that the first result was obtained by U.~Vetter in~\cite{V} and the last result above is a~generalization of a result of O.~Zariski~\cite{Z}.
G.-M.~Greuel obtained much more general results on torsion modules. See~\cite[Proposition~1.11, p.~242]{G}.

Assume that the hypersurface $S$ has an isolated singularity at the origin.
We thus have, by~combining the results of G.-M.~Greuel above and of A.G.~Aleksandrov presented in the previous section, the following:
\begin{enumerate}\itemsep=0pt
\item[$(i)$]
$\Omega_{X,O}^{q}(\log S)= \frac{{\rm d}f}{f} \wedge \Omega_{X,O}^{q-1}+\Omega_{X,O}^{q}$, $q=1,2,\ldots, n-2$,

\item[$(ii)$]
$0 \longrightarrow\frac{{\rm d}f}{f} \wedge \Omega_{X,O}^{n-2} + \Omega_{X,O}^{n-1}\longrightarrow
\Omega_{X,O}^{n-1}(\log S) \stackrel{\cdot f}\longrightarrow {\rm Tor}\big(\Omega_{S}^{n-1}\big) \longrightarrow 0$.
\end{enumerate}

Accordingly we have the following.

\begin{Proposition} Let $S = \{x \in X \,|\, f(x) = 0\}$ be a hypersurface in $X$
with an isolated sin\-gularity at $ O \in {\mathbb C}^n$. Then,
$\omega_S^{q} = \Omega_X^{q}$, $q=0,1,\ldots,n-3$ holds.
\end{Proposition}

\begin{proof}
Since $\left. {\rm res}\big(\Omega_X^{q}(\log S)\big) = \Omega_X^{q-1}\middle|\raisebox{-0.7ex}[1ex][-2ex]{{\hspace{-1mm}\;$_S$}} \right.$, $q=1,2,\ldots,n-2$, the result of A.G.~Aleksandrov presented in the last section yields the result.
\end{proof}


\section{Description via logarithmic residues}\label{sec3}

In this section, we recall results given in~\cite{TN20} to show that torsion differential forms can be des\-cri\-bed in terms of non-trivial logarithmic vector fields. We also recall basic ideas and the framework for computing non-trivial logarithmic vector fields. As an application,
we give a~method for computing logarithmic residues.

\subsection{Logarithmic vector fields}

A vector field $v$ on $X$ with holomorphic coefficients is called logarithmic
along the hyper\-sur\-face~$S$, if the holomorphic function $v(f)$ is in the ideal $(f)$ generated by $f$ in
${\mathcal O}_X$. Let ${\mathcal D}{\rm er}_{X}(-\log S) $ denote the sheaf of modules on $ X $ of logarithmic
vector fields along $S$~\cite{S}.

Let $\omega_X = {\rm d}x_1 \wedge {\rm d}x_2 \wedge \cdots \wedge {\rm d}x_n$. For a holomorphic vector field $v$,
let $i_{v}(\omega_X)$ denote the inner product of $\omega_X$ by $v$.

\begin{Proposition}
Let $S=\{x \in X \,|\, f(x) =0\}$ be a hypersurface with an isolated singularity at the origin. Then, $\Omega_{X, O}^{n-1}(\log S)$ is isomorphic to
${\mathcal D}{\rm er}_{X, O}(-\log S)$, more precisely
\begin{gather*}
\Omega_{X, O}^{n-1}(\log S) = \left\{ \frac{i_{v}(\omega_X)}{f} \middle| \ v \in {\mathcal D}{\rm er}_{X, O}(-\log S) \right\}
\end{gather*}
holds.
\end{Proposition}
\begin{proof}
Let $\beta=i_{v}(\omega_X), $ and set $ \omega = \frac{\beta}{f}$. Then, $f\omega = \beta$ is a~holomorphic differential form. Therefore, the meromorphic differential $n-1$ form
$\omega$ is logarithmic if and only if ${\rm d}f\wedge \frac{\beta}{f}$ is a~holo\-morphic differential $n$-form. Since ${\rm d}f\wedge \beta = {\rm d}f\wedge i_{v}(\omega_X) = v(f)\omega_X$, we have ${\rm d}f\wedge \frac{\beta}{f} = \frac{v(f)}{f} \omega_X$. Hence, the condition above means $v(f)$ is in the ideal $(f) \subset {\mathcal O}_{X,O}$ generated by $f$.
This completes the proof.
\end{proof}

A germ of logarithmic vector field $v$ generated over ${\mathcal O}_{X,O}$ by
\begin{gather*}
f\frac{\partial}{\partial x_i}, \qquad
i=1,2,\ldots,n, \qquad
\frac{\partial f}{\partial x_j}\frac{\partial}{\partial x_i} - \frac{\partial f}{\partial x_i}\frac{\partial}{\partial x_j}, \qquad
1 \leq i < j \leq n,
\end{gather*}
is called trivial.

\begin{Lemma}
Let $v$ be a germ of a logarithmic vector field. Then, the following conditions are equivalent:
\begin{enumerate}\itemsep=0pt
\item[$(i)$] $\displaystyle \omega = \frac{i_{v}(\omega_X)}{f}$ belongs to $\displaystyle\frac{{\rm d}f}{f} \wedge \Omega_{X,O}^{n-2}+ \Omega_{X,O}^{n-1}$,
\item[$(ii)$] $v$ is a trivial vector field.
\end{enumerate}
\end{Lemma}
\begin{proof}
The logarithmic differential form $\omega = \frac{i_{v}(\omega_X)}{f}$ is in $\Omega_{X,O}^{n-1}+\frac{{\rm d}f}{f} \wedge \Omega_{X,O}^{n-2}$ if and only if the numerator
$i_{v}(\omega_X)$ is in $f\Omega_{X,O}^{n-1}+ {\rm d}f\wedge \Omega_{X,O}^{n-2}$. The last condition is
equivalent to the triviality of~the vector field $v$, which completes the proof.
\end{proof}

For $\beta \in \Omega_{X, O}^{n-1}$, let $[\beta]$ denote the K\"ahler differential form in $\Omega_{S, O}^{n-1}$ defined by $\beta$, that is, $[\beta]$ is the equivalence class in
$\Omega_{X, O}^{n-1} /\big(f \Omega_{X, O}^{n-1} + {\rm d}f\wedge \Omega_{X, O}^{n-2}\big)$ of $\beta$.

The lemma above amount to say that, for logarithmic vector fields $v$,
$[i_{v}(\omega_X)]$ is a non-zero torsion differential form
in ${\rm Tor}\big(\Omega_{S,O}^{n-1}\big)$ if and only if $v$ is a non-trivial logarithmic vector field.

We say that germs of two logarithmic vector fields $v, v^{\prime} \in {\mathcal D}{\rm er}_{X, O}(-\log S)$ are equivalent, denoted by $ v \sim v^{\prime} $, if $ v-v^{\prime} $ is trivial.
Let ${\mathcal D}{\rm er}_{X, O}(-\log S)/{\sim}$ denote the quotient by the equivalence relation $\sim$. (See~\cite{T}.)

Now consider the following map
\begin{gather*}
 \Theta\colon\ {\mathcal D}{\rm er}_{X, O}(-\log S)/{\sim} \longrightarrow
\Omega_{X, O}^{n-1} /\big(f \Omega_{X, O}^{n-1} + {\rm d}f\wedge \Omega_{X, O}^{n-2}\big)
\end{gather*}
defined to be $ \Theta([v]) = [i_{v}(\omega_X)], $ where $ [v] $ is the equivalence class in
${\mathcal D}{\rm er}_{X, O}(-\log S)/{\sim}$ of $v$. It is easy to see that the map $\Theta$ is well-defined. We arrive at the following description of the torsion module.

\begin{Theorem}[\cite{TN20}]
The map
\begin{gather*}
\Theta\colon\ {\mathcal D}{\rm er}_{X, O}(-\log S)/{\sim} \longrightarrow {\rm Tor}\big(\Omega_{S}^{n-1}\big)
\end{gather*}
is an isomorphism.
\end{Theorem}

\subsection{Polar method}

In~\cite{T}, based on the concept of polar variety, logarithmic vector fields are studied and an effective and constructive method is considered. Here in this section, following~\cite{NT19a, T} we recall some basics and give a description of non-trivial logarithmic vector fields.

Let $S =\{x \in X \,|\, f(x) =0\} $ be a hypersurface with an isolated singularity. In what fol\-lows,
we assume that $f, \frac{\partial f}{\partial x_2}, \frac{\partial f}{\partial x_3}, \ldots, \frac{\partial f}{\partial x_{n}}$ is a regular sequence and the common \mbox{locus}
$V\big(f, \frac{\partial f}{\partial x_2}, \frac{\partial f}{\partial x_3},\ldots,\allowbreak\frac{\partial f}{\partial x_{n}}\big)\cap X$
is the origin $O$. See~\cite{NT20b} for an algorithm of testing zero-dimensiona\-lity of varieties at a point.

Let $\big(f, \frac{\partial f}{\partial x_2},$ $\frac{\partial f}{\partial x_3},\ldots,\frac{\partial f}{\partial x_{n}}\big) : \big(\frac{\partial f}{\partial x_1}\big) $ denote the ideal quotient, in the local ring ${\mathcal O}_{X,O}$, of
$\big(f, \frac{\partial f}{\partial x_2}$, $\frac{\partial f}{\partial x_3},\ldots,
\frac{\partial f}{\partial x_{n}}\big)$ by $\big(\frac{\partial f}{\partial x_1}\big)$. We have the following.

\begin{Lemma}\label{Lem2}
Let $ a(x) $ be a germ of holomorphic function in $ {\mathcal O}_{X,O}$. Then, the following are equivalent:
\begin{enumerate}\itemsep=0pt
\item[$(i)$]
$\displaystyle a(x) \in \bigg(f, \frac{\partial f}{\partial x_2}, \frac{\partial f}{\partial x_3},\ldots,\frac{\partial f}{\partial x_{n}}\bigg) : \bigg(\frac{\partial f}{\partial x_1}\bigg)$.

\item[$(ii)$]
There exists a germ of logarithmic vector field $v$ in ${\mathcal D}{\rm er}_{X, O}(-\log S)$ such that
\begin{gather*}
v= a(x)\frac{\partial }{\partial x_1} + a_2(x)\frac{\partial}{\partial x_2} + \cdots + a_{n-1}(x)\frac{\partial}{\partial x_{n-1}}+ a_n(x)\frac{\partial}{\partial x_n},
\end{gather*}
where $ a_2(x), \ldots , a_{n}(x) \in {\mathcal O}_{X, O}$.
\end{enumerate}
\end{Lemma}

Note that in~\cite{NT18,NT19a}, by utilizing local cohomology and Grothendieck local duality, an effective method of computing a set of generators over the local ring ${\mathcal O}_{X,O}$ of the module of logarithmic vector fields is given. See the next section.

\begin{Lemma}\label{Lem3}
Assume that $f, \frac{\partial f}{\partial x_2}, \frac{\partial f}{\partial x_3}, \ldots, \frac{\partial f}{\partial x_{n}}$ is a regular sequence. Let $v^{\prime}$ be a logarithmic vector fields in ${\mathcal D}{\rm er}_{X, O}(-\log S)$ of the form
\begin{gather*}
v^{\prime}= a_2(x)\frac{\partial }{\partial x_2} + a_3(x)\frac{\partial}{\partial x_3} + \cdots +a_{n}(x)\frac{\partial}{\partial x_{n}}.
\end{gather*}
Then, $v^{\prime}$ is trivial.
\end{Lemma}

Lemmas~\ref{Lem2} and~\ref{Lem3} immediately yield the following.

\begin{Proposition}
Let $f, \frac{\partial f}{\partial x_2}, \frac{\partial f}{\partial x_3}, \ldots,
\frac{\partial f}{\partial x_{n}} $ be a regular sequence. Let~$v$ be a germ of logarithmic
vector field along $S$ of the form
\begin{gather*}
v= {a_1(x)\frac{\partial }{\partial x_1} + a_2(x)\frac{\partial}{\partial x_2} + \cdots +
 a_{n-1}(x)\frac{\partial}{\partial x_{n-1}}+ a_n(x)\frac{\partial}{\partial x_n}}.
\end{gather*}
 Then, the following conditions are equivalent:
\begin{enumerate}\itemsep=0pt
\item[$(i)$] $v$ is trivial,

\item[$(ii)$] $a_1(x) \in \big(f, \frac{\partial f}{\partial x_2}, \frac{\partial f}{\partial x_3},\ldots,\frac{\partial f}{\partial x_{n}}\big)$.
\end{enumerate}
\end{Proposition}

Therefore, we have the following.

\begin{Theorem}[\cite{T}]
$ {\mathcal D}{\rm er}_{X, O}(-\log S)/{\sim} $ is isomorphic to
\begin{gather*}
\bigg(\bigg(f, \frac{\partial f}{\partial x_2}, \frac{\partial f}{\partial x_3},\ldots,\frac{\partial f}{\partial x_{n}}\bigg) : \bigg(\frac{\partial f}{\partial x_1}\bigg)\bigg)\Big/\bigg(f, \frac{\partial f}{\partial x_2}, \frac{\partial f}{\partial x_3},\ldots,\frac{\partial f}{\partial x_{n}}\bigg).
\end{gather*}

To be more precise, let $A$ be a basis as a vector space of the quotient
\begin{gather*}
\bigg(\bigg(f, \frac{\partial f}{\partial x_2}, \frac{\partial f}{\partial x_3},\ldots,\frac{\partial f}{\partial x_{n}}\bigg) : \bigg(\frac{\partial f}{\partial x_1}\bigg)\bigg)\Big/\bigg(f, \frac{\partial f}{\partial x_2}, \frac{\partial f}{\partial x_3},\ldots,\frac{\partial f}{\partial x_{n}}\bigg).
\end{gather*}
Then the corresponding logarithmic vector fields,
\begin{gather*}
v= a(x)\frac{\partial }{\partial x_1} + a_2(x)\frac{\partial}{\partial x_2} + \cdots +a_{n-1}(x)\frac{\partial}{\partial x_{n-1}}+ a_n(x)\frac{\partial}{\partial x_n}, \qquad
a(x) \in A
\end{gather*}
give rise to a basis of $ {\mathcal D}{\rm er}_{X, O}(-\log S)/{\sim}$.

\end{Theorem}


\subsection{Local cohomology and duality}

In this section, we briefly recall some basics on local cohomology and Grothendieck local duality. We give an outline for computing
non-trivial logarithmic vector fields. We refer to~\cite{TN20} for details.

Let ${\mathcal H}_{\{O\}}^n\big(\Omega_X^n\big)$ denote the local cohomology supported at the origin $O$ of the sheaf $\Omega_X^n$ of~holo\-morphic $n$-forms. Then, the stalk ${\mathcal O}_{X, O}$ and the local cohomology ${\mathcal H}_{\{O\}}^n\big(\Omega_X^n\big)$ are mutually dual as locally convex topological vector spaces.

The duality is given by the point residue pairing:
\[
{\rm Res}_{\{O\}}(*, *)\colon\ {\mathcal O}_{X,O} \times {\mathcal H}_{\{O\}}^n\big(\Omega_X^n\big) \longrightarrow {\mathbb C}.
\]

Let $W_{\Gamma(f)}$ denote the set of local cohomology classes in ${\mathcal H}_{\{O\}}^n\big(\Omega_X^n\big)$ that are annihilated by
$f$, $\frac{\partial f}{\partial x_2}, \frac{\partial f}{\partial x_{3}}, \dots, \frac{\partial f}{\partial x_{n}}$:
\[
W_{\Gamma(f)} = \SetDef{ \varphi \in {\mathcal H}_{\{O\}}^n\big(\Omega_X^n\big)}{f\varphi = \frac{\partial f}{\partial x_2}\varphi = \cdots = \frac{\partial f}{\partial x_{n}}\varphi = 0}.
\]
Then, a complex analytic version of Grothendieck local duality on residue implies that the pairing
\begin{gather*}
{\mathcal O}_{X,O}\Big/\bigg(f, \frac{\partial f}{\partial x_2}, \frac{\partial f}{\partial x_3},\ldots,\frac{\partial f}{\partial x_{n}}\bigg) \times W_{\Gamma(f)} \longrightarrow {\mathbb C}
\end{gather*}
is non-degenerate.

Let $\mu(f)$ and $\mu(f|_{H_{x_1}})$ denote the Milnor number of $f$ and that of a hyperplane section $f|_{H_{x_1}}$ of $f$, where $f|_{H_{x_1}}$ is the restriction of $f$ to the hyperplane $H_{x_1} =\{ x \in X \,|\, x_1=0\}$. Then, the classical L\^e--Teissier formula~\cite{L,Te73} and the Grothendieck local duality imply the following:
\[
\dim_{{\mathbb C}}W_{\Gamma(f)} = \mu(f) + \mu(f|_{H_{x_1}}).
\]

Let $\gamma\colon W_{\Gamma(f)} \longrightarrow W_{\Gamma(f)}$ be a map defined by
$\gamma(\varphi)=\frac{\partial f}{\partial x_1}\ast\varphi$ and let $W_{\Gamma(f)}$ be the image of the map $\gamma$:
\[
W_{\Delta(f)} = \SetDef{ \frac{\partial f}{\partial x_1}\ast\varphi}{\varphi \in W_{\Gamma(f)}}.
\]

Let $ {\rm Ann}_{{\mathcal O}_{X,O}}(W_{\Delta(f)}) $ be the annihilator in $ {\mathcal O}_{X,O} $ of the set $ W_{\Delta(f)} $ of local cohomology classes.
We have the following.

\begin{Lemma}[\cite{T}]
${\rm Ann}_{{\mathcal O}_{X,O}}(W_{\Delta(f)}) = \big(f, \frac{\partial f}{\partial x_2},
\frac{\partial f}{\partial x_3},\ldots,
\frac{\partial f}{\partial x_{n}}\big) : \big(\frac{\partial f}{\partial x_1}\big)$.
\end{Lemma}

\begin{proof}
See~\cite{NT16a,T,TNN}.
\end{proof}

Recall that the ideal quotient $\big(f, \frac{\partial f}{\partial x_2}, \frac{\partial f}{\partial x_3},\ldots,\frac{\partial f}{\partial x_{n}}\big) : \big(\frac{\partial f}{\partial x_1}\big)$ is coefficient ideal w.r.t.~$\frac{\partial}{\partial x_1} $ of~loga\-rithmic vector fields along $S$. The lemma above says that the coefficient ideal can be~described in terms of local cohomology $ W_{\Delta(f)}$.

Let $W_{T(f)}$ be the kernel of the map $\gamma$. By definition we have
\[
W_{T(f)} = \SetDef{\varphi \in {\mathcal H}_{\{O\}}^n\big(\Omega_X^n\big)}{f\varphi = \frac{\partial f}{\partial x_1}\varphi = \frac{\partial f}{\partial x_2}\varphi = \cdots = \frac{\partial f}{\partial x_{n}}\varphi = 0}.
\]
Since the pairing
\begin{gather*}
\displaystyle{\mathcal O}_{X,O}\Big/\!\bigg(f, \frac{\partial f}{\partial x_1} \frac{\partial f}{\partial x_2}, \frac{\partial f}{\partial x_3},\ldots,\frac{\partial f}{\partial x_{n}}\bigg) \times W_{T(f)} \longrightarrow {\mathbb C}
\end{gather*}
is non-degenerate by Grothendieck local duality,
$ \dim_{{\mathbb C}}(W_{T(f)}) $ is equal to
\[
\tau = \dim_{{\mathbb C}}\bigg( {\mathcal O}_{X,O}\Big/\!\bigg(f, \frac{\partial f}{\partial x_1} \frac{\partial f}{\partial x_2}, \frac{\partial f}{\partial x_3},\ldots,\frac{\partial f}{\partial x_{n}}\bigg)\bigg),
\]
the Tjurina number.

From the exactness of the sequence
\[
0 \longrightarrow W_{T(f)} \longrightarrow W_{\Gamma(f)} \longrightarrow W_{\Delta(f)} \longrightarrow 0,
\]
we have
\[
\dim_{{\mathbb C}}W_{\Delta(f)} = \mu(f) - \tau(f) + \mu(f|_{H_{x_1}}).
\]

The argument above also implies the following.

\begin{Corollary}[\cite{T}]
\[
\dim_{{\mathbb C}}\big({\mathcal D}{\rm er}_{X, O}(-\log S)/{\sim}\big) = \tau.
\]
\end{Corollary}

Notice that the dimension of $W_{\Delta(f)}$ that measures the way of vanishing of coefficients of~loga\-rithmic vector fields depends on the choice of a~system of coordinates, or a~hyperplane. In~order to analyze complex analytic properties of logarithmic vector fields, as we observed in~\cite{T}, it~is important to select an appropriate system of coordinates or a~generic hyperplane. We return to~this issue afterwards at the end of this section.

Now let
$\displaystyle H_{[O]}^n({\mathcal O}_X) = \lim_{k \to \infty} {\rm Ext}_{{\mathcal O}_X}^n \big({\mathcal O}_{X,O}/(x_1,x_2,\ldots,x_n)^k, {\mathcal O}_X\big)$
be the sheaf of algebraic local cohomology and let
\begin{gather*}
H_{\Gamma(f)} = \SetDef{ \phi \in H_{[O]}^n({\mathcal O}_X) }{f\phi = \frac{\partial f}{\partial x_2}\phi = \cdots = \frac{\partial f}{\partial x_{n}}\phi = 0},
\\
H_{\Delta(f)} = \SetDef{\frac{\partial f}{\partial x_1}\phi}{\phi\in H_{\Gamma(f)}}.
\end{gather*}
Then, the following holds
\[
W_{\Gamma(f)} = \{\phi \cdot \omega_X \,|\, \phi \in H_{\Gamma(f)}\}, \qquad
W_{\Delta(f)} = \{\phi \cdot \omega_X \,|\, \phi \in H_{\Delta(f)}\}.
\]

In~\cite{TNN}, algorithms for computing algebraic local cohomology classes and some relevant algo\-rithms are given.
Accordingly, $ H_{\Gamma(f)}, H_{\Delta(f)} $ are computable. Note also that a~standard basis of~the ideal quotient
$\big(f, \frac{\partial f}{\partial x_2}, \frac{\partial f}{\partial x_3},\ldots,
\frac{\partial f}{\partial x_{n}}\big) : \big(\frac{\partial f}{\partial x_1}\big)$ can be computed by using $ H_{\Delta(f)} $ in an efficient manner~\cite{TNN}.

Now we present an outline of a method for constructing a basis, as a vector space, of the quotient space
$\big(\big(f, \frac{\partial f}{\partial x_2}, \frac{\partial f}{\partial x_3},\ldots,
\frac{\partial f}{\partial x_{n}}\big) : \big(\frac{\partial f}{\partial x_1}\big)\big)/\big(f, \frac{\partial f}{\partial x_2}, \frac{\partial f}{\partial x_3},\ldots,
\frac{\partial f}{\partial x_{n}}\big)$.

We fix a term ordering $\succ$ on $H_{[O]}^n({\mathcal O}_X)$ and its inverse term ordering $\succ^{-1}$
on the local ring~${\mathcal O}_{X,O}$.
\begin{enumerate}\itemsep=0pt
\setlength{\leftskip}{0.45cm}

\item[{\it Step} 1:] Compute a basis $\Phi_{\Gamma(f)} $ of $ H_{\Gamma(f)}$.

\item[{\it Step} 2:] Compute a monomial basis $M_{\Gamma(f)}$ of the quotient space ${\mathcal O}_{X,O}/
\big(f, \frac{\partial f}{\partial x_2}, \frac{\partial f}{\partial x_3},\ldots,
\frac{\partial f}{\partial x_{n}}\big)$, with respect to $\succ^{-1}$, by using $\Phi_{\Gamma(f)}$.

\item[{\it Step} 3:] Compute $\frac{\partial f}{\partial x_n}\phi$ of each $\phi \in \Phi_{\Gamma(f)}$ and compute a basis $\Phi_{\Delta(f)}$ of $H_{\Delta(f)}$.

\item[{\it Step} 4:] Compute a standard basis ${\rm SB}$ of the ideal ${\rm Ann}_{{\mathcal O}_{X, O}}(H_{\Delta(f)})$ by using $\Phi_{\Delta(f)}$.

\item[{\it Step} 5:] Compute the normal form ${\rm NF}_{\succ^{-1}}\big(x^{\lambda} s(x)\big)$ of $x^{\lambda}s(x)$ for $x^{\lambda} \in M_{\Gamma(f)}, s(x) \in {\rm SB}$.

\item[{\it Step} 6:] Compute a basis $ {\rm A} $, as a vector space, of
${\rm Span}_{{\mathbb C}}\big\{ {\rm NF}_{\succ^{-1}}(x^{\lambda} s(x)) \,|\, x^{\lambda} \in M_{\Gamma(f)}$,\linebreak \mbox{$s(x) \in {\rm SB} \big\}$}.
\end{enumerate}
Then, we have the following:
\begin{gather*}
{\rm Span}_{{\mathbb C}}({\rm A}) \cong \bigg(\bigg(f, \frac{\partial f}{\partial x_2}, \frac{\partial f}{\partial x_3},\ldots,
\frac{\partial f}{\partial x_{n}}\bigg) : \bigg(\frac{\partial f}{\partial x_1}\bigg)\bigg)\Big/\bigg(f, \frac{\partial f}{\partial x_2}, \frac{\partial f}{\partial x_3},\ldots,
\frac{\partial f}{\partial x_{n}}\bigg).
\end{gather*}

Note that, by utilizing algorithms given in~\cite{NT17a}, the method proposed above can be extended to treat parametric cases,
the case where the input data contain parameters.

In order to obtain non-trivial logarithmic vector fields, it is enough to do the following.

For each $a(x) \in A$, compute $a_2(x), a_3(x),\ldots,a_n(x), b(x) \in {\mathcal O}_{X,O}$, such that
\begin{gather*}
a(x)\frac{\partial f}{\partial x_1} + a_2(x)\frac{\partial f}{\partial x_2} + \cdots + a_{n-1}(x)\frac{\partial f}{\partial x_{n-1}}+ a_n(x)\frac{\partial f}{\partial x_n} - b(x)f(x) =0.
\end{gather*}

{\samepage
Then,
\begin{gather*}
a(x)\frac{\partial }{\partial x_1} + a_2(x)\frac{\partial }{\partial x_2} + \cdots + a_{n-1}(x)\frac{\partial }{\partial x_{n-1}}+ a_n(x)\frac{\partial }{\partial x_n}, \qquad a(x) \in A
\end{gather*}
gives rise to the desired set of non-trivial logarithmic vector fields.

}

The step above can be executed efficiently by using an algorithm described in~\cite{NT16b}. See also~\cite{TN20} for details.

Before ending this section, we turn to the issue on the genericity. For this purpose, let us recall a result of B. Teissier on this subject.

Let $ p^{\prime}= (p_1^{\prime}, p_2^{\prime}, \dots, p_n^{\prime}) $ be a non-zero vector and let $ [p^{\prime}] $ denote
the corresponding point in~the projective space $ {\mathbb P}^{n-1}$. We identify the hyperplane
\[
H_{p^{\prime}} = \big\{ (x_1,x_2,\dots, x_n) \in {\mathbb C}^{n} \,|\, p_1^{\prime}x_1+p_2^{\prime}x_2+ \cdots + p_n^{\prime}x_n =0 \big\}
\]
with the point $ [p^{\prime}] $ in $ {\mathbb P}^{n-1}$.
In~\cite{Te73, Te77}, B.~Teissier introduced an invariant $ \mu^{(n-1)}(f) $ as
\[
\mu^{(n-1)}(f) = \min_{[p^{\prime}] \in {\mathbb P}^{n-1}} \mu(f|_{H_{p^{\prime}}}),
\]
where $ f|_{H_{p^{\prime}}} $ is the restriction of $f$ to $ H_{p^{\prime}} $ and
$ \mu(f|_{H_{p^{\prime}}}) $ is the Milnor number at the origin $O$ of~the~hyperplane section $ f|_{H_{p^{\prime}}} $ of $f $.
He also proved that the set
\[
U=\big\{[p^{\prime}] \in {\mathbb P}^{n-1} \,|\, \mu(f|_{H_{p^{\prime}}}) = \mu^{(n-1)}(f)\big\}
\]
is a Zariski open dense subset of $ {\mathbb P}^{n-1}$.

Accordingly, in order to obtain good representations of logarithmic vector fields, it is desirable to use a generic system of coordinate or a generic hyperplane $H_{p^{\prime}}$ that satisfies the condition $\mu(f|_{H_{p^{\prime}}}) = \mu^{(n-1)}(f)$.

In a previous paper~\cite{NT18b}, methods for computing limiting tangent spaces were studied and an~algorithm of computing $ \mu(f|_{H_{p^{\prime}}})$, $p^{\prime} \in {\mathbb P}^{n-1} $ was given. In~\cite{NT17b,NT19b}, more effective algorithms for computing $ \mu^{(n-1)} $ were given. Utilizing the results in~\cite{NT17b,NT19b}, an effective method for compu\-ting logarithmic vector fields that takes care of the genericity condition is designed in~\cite{NT19a, TN20}. See~also~\cite{TN20a} for related results.

\subsection{Regular meromorphic differential forms}

Now we are ready to consider a method for computing regular meromorphic differential forms. For simplicity, we first consider a 3-dimensional case. Assume that a non-trivial logarithmic vector field~$v$ is given:
\begin{gather*}
v=a_1(x)\frac{\partial}{\partial x_1} + a_2(x)\frac{\partial}{\partial x_2}+a_3(x)\frac{\partial}{\partial x_3}.
\end{gather*}

Let $v(f)=b(x)f(x)$ and $\beta=i_v(\omega_X) $, where $\omega_X={\rm d}x_1\wedge {\rm d}x_2\wedge {\rm d}x_3$. We have $\beta=a_1(x){\rm d}x_2\wedge {\rm d}x_3-a_2(x){\rm d}x_1\wedge {\rm d}x_3+a_3(x){\rm d}x_1\wedge {\rm d}x_2$.
We introduce differential forms~$ \xi$ and~$\eta$ as
\begin{gather*}
\xi=-a_2(x){\rm d}x_3+a_3(x){\rm d}x_2, \qquad
\eta =b(x){\rm d}x_2\wedge {\rm d}x_3.
\end{gather*}

Let $g(x)=\frac{\partial f}{\partial x_1}$. Then, the following holds
\begin{gather*}
g(x)\beta= {\rm d}f\wedge \xi + f(x)\eta.
\end{gather*}

Accordingly, the logarithmic differential form $ \omega=\frac{\beta}{f} $ satisfies
\begin{gather*}
g(x)\omega=\frac{{\rm d}f}{f} \wedge \xi + \eta.
\end{gather*}

We may assume that the coordinate system $(x_1,x_2,x_3)$ is generic~\cite{NT19a} and $g(x)$ satisfies the condition $(a)$, $(b)$ of $(iii)$ in Definition~\ref{def1}.

Since $ g(x) = \frac{\partial f}{\partial x_1}, $ we have, by definition, the following:
\begin{gather*}
{\rm res}\bigg(\frac{\beta}{f}\bigg) = \frac{\xi}{\frac{\partial f}{\partial x_1}}\Bigg|_S.
\end{gather*}
Notice that the differential form $ \xi$ above is directly defined from the coefficients of the logarithmic vector field $v$.

\begin{Proposition}\label{pro4}
 Let $S=\{ x \in X \,|\, f(x)=0 \} $ be a hypersurface with an isolated singularity at the origin $ O \in X \subset {\mathbb C}^n$.
Assume that the coordinate system $ (x_1,x_2,\ldots,x_n) $ is generic so that
$\big(f, \frac{\partial f}{\partial x_2}, \frac{\partial f}{\partial x_3},\ldots,\frac{\partial f}{\partial x_{n}}\big)$ is a regular sequence and $ g(x) =\frac{\partial f}{\partial x_1} $ satisfies the condition~$(a)$,~$(b)$ of $(iii)$ in Definition~$\ref{def1}$.
Let
\begin{gather*}
v=a_1(x)\frac{\partial}{\partial x_1} + a_2(x)\frac{\partial}{\partial x_2}+ \cdots +a_n(x)\frac{\partial}{\partial x_n}
\end{gather*}
be a germ of non-trivial logarithmic vector field along $S$. Let $ v(f)=b(x)f(x)$, $\beta=i_v(\omega_X)$. Let~$\xi$,~$\eta$ denote the differential form defined to be
\begin{gather*}
\xi = -a_2(x){\rm d}x_3 \wedge {\rm d}x_4 \wedge \cdots \wedge {\rm d}x_n +a_3(x){\rm d}x_2 \wedge {\rm d}x_4 \wedge \cdots \wedge {\rm d}x_n - \cdots
\\\phantom{\xi =-}{}
 +(-1)^{(n+1)}a_n(x){\rm d}x_2 \wedge {\rm d}x_3 \wedge \cdots \wedge {\rm d}x_{n-1}, \\
\eta = \phantom{-}b(x) {\rm d}x_2 \wedge {\rm d}x_3 \wedge \cdots \wedge {\rm d}x_n.
\end{gather*}
Then,
\begin{gather*}
g(x)\frac{\beta}{f} = \frac{{\rm d}f}{f}\wedge \xi + \eta\qquad
\text{and} \qquad
{\rm res}\bigg(\frac{\beta}{f}\bigg) = \frac{\xi}{\frac{\partial f}{\partial x_1}}\Bigg|_S
\end{gather*}
 hold.
 \end{Proposition}

Note that, in 1984, M.~Kersken~\cite{K84} obtained related results on regular meromorphic differential forms. The statement in Proposition~\ref{pro4} above is a refinement a result of M.~Kersken.

\begin{Theorem}\label{Th6}
Let $S=\{ x \in X \,|\, f(x)=0\} $ be a hypersurface with an isolated singularity at the origin $ O \in X \subset {\mathbb C}^n$. Let $V=\{v_1, v_2,\ldots, v_{\tau}\} $ be a set of non-trivial logarithmic vector fields such that the class $[v_1], [v_2], \dots, [v_{\tau}]$ constitute a~basis of the vector space ${\mathcal D}{\rm er}_{X,O}(-\log S)/{\sim}$, where~$\tau$ stands for the Tjurina number of $f$. Let $\xi_1, \xi_2, \ldots, \xi_{\tau}$ be the differential forms correspond to
$v_1, v_2, \ldots, v_{\tau}$ defined in Proposition~$\ref{pro4}$.

Then, any logarithmic residue in ${\rm res}\big(\Omega^{n-1}(\log S)\big)$, or a regular meromorphic differential form~$\gamma$ in $\omega_S^{n-2}$ can be represented as
\begin{gather*}
\gamma = \bigg(\frac{1}{\frac{\partial f}{\partial x_1}}(c_1\xi_1+c_2\xi_2+ \cdots +c_{\tau}\xi_{\tau})\bigg)\bigg|_S + \alpha,
\end{gather*}
where $c_i \in {\mathbb C}$, $i=1,2,\dots, \tau$, and $ \left. \alpha \in \Omega_X^{n-2}\middle|\raisebox{-0.7ex}[1ex][-2ex]{{\hspace{-1mm}\;$_S$}}\right. $.
\end{Theorem}

\section{Examples}\label{sec4}

In this section, we give examples of computation for illustration. Data is an extraction from~\cite{TN20}. Let $f_0(z,x,y)=x^3+y^3+z^4 $ and let $ f_t(z,x,y)=f_0(z,x,y)+txyz^2$,
where $t$ is a deformation parameter. We regard $z$ as the first variable. Then, $f_0$ is a weighted homogeneous polynomial with respect to a weight vector $(3,4,4)$ and $f_t$ is a $\mu$-constant deformation of $f_0$, called $U_{12}$ singularity. The Milnor number
$\mu(f_t)$ of $U_{12}$ singularity is equal to 12. In contrast, the Tjurina number
$\tau(f_t)$ depends on the parameter $t$. In fact, if $t=0$, then $\tau(f_0) = 12$ and if $t \ne 0$, then $\tau(f_t) =11$. In the computation, we fix a term order $\succ^{-1}$ on ${\mathcal O}_{X,O}$ which is compatible with the weight vector $(3,4,4)$.

We consider these two cases separately.

\begin{Example}[weighted homogeneous $ U_{12} $ singularity]
Let $ f_0(z,x,y) = x^3+y^3+z^4$. Then, $ \mu(f_0)=\tau(f_0)=12$.
The monomial basis $ {\rm M} $ with respect to the term ordering $ \succ^{-1} $ of the quotient space
$ {\mathcal O}_{X,O}/(f_0, \frac{\partial f_0}{\partial x}, \frac{\partial f_0}{\partial y}) $ is
\begin{gather*}
{\rm M} = \big\{ x^i y^j z^k \,|\, i=0,1, \ j=0,1, \ k=0,1,2,3 \big\}.
\end{gather*}
The standard basis $ {\rm Sb}$ of the ideal quotient
$\big(f_0, \frac{\partial f_0}{\partial x}, \frac{\partial f_0}{\partial y}\big): \big(\frac{\partial f_0}{\partial z}\big)$ is
\begin{gather*}
{\rm Sb} = \big\{ x^2, y^2, z \big\}.
\end{gather*}
The normal form in ${\mathcal O}_{X,O}/\big(f_0, \frac{\partial f_0}{\partial x}, \frac{\partial f_0}{\partial y}\big)$ of $x^2$, $y^2$ and $z$ are
\begin{gather*}
{\rm NF}_{\succ^{-1}}\big(x^2\big) = {\rm NF}_{\succ^{-1}}\big(y^2\big) = 0,\qquad
 {\rm NF}_{\succ^{-1}}(z) = z.
\end{gather*}
Therefore,
$ {\rm A} = \{x^iy^jz^k \,|\, i=0,1, \ j=0,1, \ k=1,2,3 \}$. Notice that ${\rm A}$ consists of $12$ elements. It is easy to see that the Euler vector field
\begin{gather*}
v= 4x\frac{\partial}{\partial x} + 4y\frac{\partial}{\partial y} +3z\frac{\partial}{\partial z}
\end{gather*}
that corresponds to the element $ z \in {\rm A} $ is a non-trivial logarithmic vector field. Therefore, the torsion module of the hypersurface
$ S_0 = \big\{ (x, y,z) \,|\, x^3+y^3+z^4=0 \big\} $ is given by
\begin{gather*}
{\rm Tor}\big(\Omega_{S_0}^2\big) = \big\{ x^iy^jz^k i_v(\omega_X) \,|\, i=0,1, \ j=0,1, \ k=1,2,3 \big\},
\end{gather*}
where $ \omega_X = {\rm d}z \wedge {\rm d}x \wedge {\rm d}y$.

Let $ \xi = -4x{\rm d}y+4y{\rm d}x$. Then $\left. {\rm res}\big(\frac{i_v(\omega_X)}{f}\big) = \frac{\xi}{4z^3}\middle|\raisebox{-0.7ex}[1ex][-2ex]{{\hspace{-1mm}\;$_S$}}\right.$. Computation of other logarithmic residues are same.
\end{Example}

The following is also an extraction from~\cite{TN20}.

\begin{Example}[semi quasi-homogeneous $U_{12}$ singularity]
Let $f(x,y,z) = x^3+y^3+z^4+txyz^2$, $t\ne 0$. Then, $\mu(f)=12$, $\tau(f)=11$ and $\mu(f|H_{z})=4$.
We have $\dim_{{\mathbb C}}H_{\Gamma(f)} = 16$, $\dim_{{\mathbb C}}H_{\Delta(f)}=5$.
Let $\succ$ be a term ordering on $H_{[O]}^{3}({\mathcal O}_X)$ which is compatible with
the weight vector $(4,4,3)$.

A basis $\Phi_{{\Gamma(f)}}$ of $H_{{\Gamma(f)}}$ is given by
\begin{gather*}
\left\{
\begin{bmatrix} 1 \\ x y z \end{bmatrix},
\begin{bmatrix} 1 \\ x y z^{2} \end{bmatrix},
\begin{bmatrix} 1 \\ x^2 y z \end{bmatrix},
\begin{bmatrix} 1 \\ x y^2 z \end{bmatrix},
\begin{bmatrix} 1 \\ x y z^{3} \end{bmatrix},
\begin{bmatrix} 1 \\ x^{2} y z^{2} \end{bmatrix},
\begin{bmatrix} 1 \\ x y^{2} z^{2} \end{bmatrix},
\begin{bmatrix} 1 \\ x^{2} y^{2} z \end{bmatrix},
\begin{bmatrix} 1 \\ x y z^{4}\end{bmatrix},\right.
\\
\hphantom{\left\{\right.}
\begin{bmatrix} 1 \\ x^{2} y z^{3} \end{bmatrix} -\dfrac{t}{3}
\begin{bmatrix} 1 \\ x y^{3} z\end{bmatrix},
\begin{bmatrix} 1 \\ x y^{2} z^{3} \end{bmatrix} -\dfrac{t}{3}
\begin{bmatrix} 1 \\ x~3 y z \end{bmatrix},
\begin{bmatrix} 1 \\ x^{2} y^{2} z^{2}\end{bmatrix},
\begin{bmatrix} 1 \\ x^{2} y z^{4}\end{bmatrix}-\dfrac{t}{3}
\begin{bmatrix} 1 \\ x y^{3} z^{2}\end{bmatrix},
\\
\hphantom{\left\{\right.}
\begin{bmatrix} 1 \\ x y^{2} z^{4}\end{bmatrix} -\dfrac{t}{3}
\begin{bmatrix} 1 \\ x^{3} y z^{2}\end{bmatrix},
\begin{bmatrix} 1 \\ x^{2} y^{2} z^{3}\end{bmatrix} -\dfrac{t}{3}
\begin{bmatrix} 1 \\ x^{4} y z\end{bmatrix} -\dfrac{t}{3}
\begin{bmatrix} 1 \\ x y^{4} z\end{bmatrix} -\dfrac{t}{3}
\begin{bmatrix} 1 \\ x y z^{5}\end{bmatrix},
\\
\hphantom{\left\{\right.}
\left.
\begin{bmatrix} 1 \\ x^{2} y^{2} z^{4}\end{bmatrix} -\dfrac{t}{3}
\begin{bmatrix} 1 \\ x^{4} y z^{2} \end{bmatrix} -\dfrac{t}{3}
\begin{bmatrix} 1 \\ x y^{4} z^{2}\end{bmatrix}-\dfrac{t}{3}
\begin{bmatrix} 1 \\ x y z^{6}\end{bmatrix}\right\}.
\end{gather*}

The monomial basis ${\rm M} $ with respect to the term ordering $\succ^{-1}$ of the quotient
${\mathcal O}_{X,O}/\big(f,$ $\frac{\partial f}{\partial x}, \frac{\partial f}{\partial y}\big)$ is
\begin{gather*}
{\rm M} = \big\{ x^i y^j z^k \,|\, i=0,1, \ j=0,1, \ k=0,1,2,3 \big\}.
\end{gather*}

A basis $ \Phi_{\Delta(f)} $ of $ H_{\Delta(f)} $ is given by
\[
\left\{
\begin{bmatrix} 1 \\ x y z \end{bmatrix},
\begin{bmatrix} 1 \\ x y z^{2}\end{bmatrix},
\begin{bmatrix} 1 \\ x^2 y z \end{bmatrix},
\begin{bmatrix} 1 \\ x y^2 z\end{bmatrix},
\begin{bmatrix} 1 \\ x^{2} y^{2} z\end{bmatrix}+ \dfrac{t}{6}
\begin{bmatrix} 1 \\ x y z^{3}\end{bmatrix} \right\}.
\]

We see from this data that the standard basis of the ideal quotient $\big(f, \frac{\partial f}{\partial x}, \frac{\partial f}{\partial y}\big) : \big(\frac{\partial f}{\partial z}\big)$
in the local ring ${\mathcal O}_{X,O}$ is
\begin{gather*}
{\rm Sb} = \left\{z^2-\dfrac{t}{6}xy, xz, yz, x^2, y^2\right\}.
\end{gather*}
From $ {\rm Sb} $ and $ {\rm M} $, we have
\begin{gather*}
{\rm A} =\left\{z^2-\dfrac{t}{6}xy, xz, yz, z^3, xz^2, yz^2, xyz, xz^3, yz^3, xyz^2, xyz^3\right\}.
\end{gather*}

These 11 elements in $ {\rm A} $ are used to construct non-trivial logarithmic vector fields and regular meromorphic differential forms.
We give the results of computation.
\begin{enumerate}\itemsep=0pt
\item[$(i)$] Let $ a=6z^2-txy$. Then,
\begin{gather*}
v=\frac{d_1}{27+t^3z^2}\frac{\partial}{\partial x} +\frac{d_2}{27+t^3z^2}\frac{\partial}{\partial y} +\big(6z^2-txy\big)\frac{\partial}{\partial z}
\end{gather*}
is a non-trivial logarithmic vector field, where
\begin{gather*}
d_1= 216xz-6t^2y^2z-2t^4x^2yz, \qquad
d_2=216yz+24t^2x^2z+10t^3yz^3-2t^4xy^2z.
\end{gather*}

\item[$(ii)$] Let $ a=xz$. Then,
\begin{gather*}
v=\frac{d_1}{27+t^3z^2}\frac{\partial}{\partial x} +\frac{d_2}{27+t^3z^2}\frac{\partial}{\partial y} +xz\frac{\partial}{\partial z}
\end{gather*}
is a non-trivial logarithmic vector field, where
\begin{gather*}
d_1=36x^2-6yz^2-6t^2xy^2, \qquad d_2=36xy+2t^2x^3-4t^2y^3-2t^2z^4.
\end{gather*}
\end{enumerate}
We omit the other nine cases. As described in Theorem~\ref{Th6}, regular meromorphic differential forms can be constructed directly from these data.
\end{Example}

\section{Brieskorn formula}\label{sec5}

In 1970, B.~Brieskorn studied the monodromy of Milnor fibration and developed the theory of~Gauss--Manin connection~\cite{Br}. He proved the regularity of the connection and proposed an~alge\-braic framework for computing the monodromy via Gauss--Manin connection. He gave in~particular a basic formula, now called Brieskorn formula, for computing Gauss--Manin connection.

We show in this section a link between Brieskorn formula, torsion differential forms and logarithmic vector fields. We present an alternative method for computing non-trivial logarithmic vector fields. The resulting algorithm can be used as a basic tool for studying Gauss--Manin connections. We also present some examples for illustration.

\subsection{Brieskorn lattice and Gauss--Manin connection}

We briefly recall some basics on Brieskorn lattice and Brieskorn formula. We refer to~\cite{BS, Br, Schu}.
Let $f(x) $ be a holomorphic function on $X$ with an isolated singularity at the origin $ O \in X, $ where $X$ is an open neighborhood of $ O $ in $ {\mathbb C}^n$.
Let
\begin{gather*}
H_{0}^{\prime} = \Omega_{X, O}^{n-1}/\big({\rm d}f\wedge \Omega_{X,O}^{n-2} + {\rm d} \Omega_{X,O}^{n-2}\big), \qquad H_{0}^{\prime\prime}=\Omega_{X,O}^{n}/{\rm d}f\wedge {\rm d}\Omega_{X,O}^{n-2}.
\end{gather*}

Then, ${\rm d}f\wedge H_{0}^{\prime} \subset H_{0}^{\prime\prime}$.
A map $D\colon {\rm d}f\wedge H_{0}^{\prime} \longrightarrow H_{0}^{\prime\prime} $ is defined as follows:
\begin{gather*}
D({\rm d}f\wedge \varphi) = [{\rm d}\varphi], \qquad \varphi \in \Omega_{X,O}^{n-1}.
\end{gather*}

Let $ \varphi= \sum_{i=1}^{n} (-1)^{i+1}h_i(x){\rm d}x_1 \wedge {\rm d}x_2 \wedge \cdots \wedge {\rm d}x_{i-1} \wedge {\rm d}x_{i+1} \wedge \cdots \wedge {\rm d}x_n$. Then
\begin{gather*}
{\rm d}f\wedge \varphi = \Bigg(\sum_{i=1}^{n} h_i(x)\frac{\partial f}{\partial x_i}\Bigg) \omega_X,
\end{gather*}
where $ \omega_X={\rm d}x_1\wedge {\rm d}x_2 \wedge \cdots \wedge {\rm d}x_n$. Therefore in terms of the coordinate we have the following, known as Brieskorn formula
\begin{gather*}
D({\rm d}f\wedge \varphi) = \Bigg(\sum_{i}^{n} \frac{\partial h_i}{\partial x_i}\Bigg)\omega_X.
\end{gather*}

\begin{Example}
 Let $f(x,y)=x^2-y^3 $ and $ S=\{ (x,y) \in X \,|\, f(x,y)=0 \}$ where $X \subset {\mathbb C}^2$ is an~open neighborhood of the origin $O$. The Jacobi ideal $J$ of $f$ is $\big(x, y^2\big) \subset {\mathcal O}_{X,O}$ and $M=\{1, y\}$ is~a~monomial basis of the quotient
${\mathcal O}_{X,O}/J$. Let $\tau$ denote the Tjurina number. Then, since $f$ is a~weighted homogeneous polynomial, we have $\tau=\mu=2$ (see Example~\ref{Ex1}).

Let $v=\frac{1}{6}\big(3x\frac{\partial}{\partial x}+2y\frac{\partial}{\partial y}\big) $ be the Euler vector field. Then, $v$ is logarithmic along $ S$.

Let $\beta=i_v(\omega_X)$. Then, $ \beta=\frac{1}{6}(3x{\rm d}y-2y{\rm d}x)$.
Since $v(f)=f$, we have $ {\rm d}f\wedge \beta = f\omega_X, $ where $ \omega_X={\rm d}x \wedge {\rm d}y$.
By Brieskorn formula, we have
\begin{gather*}
D(f\omega_X) = D({\rm d}f\wedge \beta) =\dfrac{5}{6}\omega_X.
\end{gather*}

Note that the formula above is equivalent $ {\rm d} \big(\frac{\beta}{f^{\lambda}}\big) =0$, with $\lambda = \frac{5}{6}$.

Likewise, for $ y\beta$, we have $ {\rm d}f\wedge (y\beta) = f(x,y)y\omega_X $ and
\begin{gather*}
D(f(x,y)y\omega_X) = D({\rm d}f\wedge (y\beta)) = \dfrac{7}{6}y\omega_X,
\end{gather*}
which is equivalent to
${\rm d} \big(\frac{y\beta}{f^{\lambda}}\big) =0$, with $\lambda = \frac{7}{6}$.

Since $Df=fD+1$ as operators, we have
\[
fD(\omega_X)=-\dfrac{1}{6}\omega_X, \qquad
fD(y\omega_X)=\dfrac{1}{6}y\omega_X.
\]
Notice that $\beta$, $y\beta$ are non-zero torsion differential forms in $ \Omega_S^{1} $ and $v$, $yv$ are non-trivial logarithmic vector fields along $S$. Note also that $yv(f)=yf$.
Notably, Brieskorn formula described in terms of differential forms can be rewritten in terms of non-trivial logarithmic vector fields $v$ and $yv$ which satisfy $v(f)=f$ and $yv(f)=yf$ respectively.
\end{Example}

Let $S=\{x \in X \,|\, f(x)=0 \}$ be the hypersurface with an isolated singularity at the origin $O \in X$ defined by $f$.
Consider, for instance, a trivial vector field $v^{\prime}=\frac{\partial f}{\partial x_2}\frac{\partial}{\partial x_1}-\frac{\partial f}{\partial x_1}\frac{\partial}{\partial x_2}$. Since $v^{\prime}(f)=0$ and $\frac{\partial}{\partial x_1}\big(\frac{\partial f}{\partial x_2}\big)+\frac{\partial}{\partial x_2}\big({-}\frac{\partial f}{\partial x_1}\big) =0$ hold, we have a trivial relation $ D((0 \cdot \omega_X) = 0 \cdot \omega_X$.
It~is easy to see in general that, from a trivial vector field Brieskorn formula only gives the trivial relation.

The observation above leads the following.

\begin{Proposition}\label{pro5}
Let $ S=\{x \in X \,|\, f(x)=0 \} $ be a hypersurface with an isolated singularity at the origin $O \in X, $ where $ X \subset {\mathbb C}^n$. Let
\begin{gather*}
v=a_1(x)\frac{\partial}{\partial x_1} + a_2(x)\frac{\partial}{\partial x_2}+ \cdots +a_n(x)\frac{\partial}{\partial x_n}
\end{gather*}
be a germ of non-trivial logarithmic vector field along $S$.
 Let
$ v(f) = b(x)f(x) $
Then,
\begin{gather*}
D(f(x)b(x)\omega_X) = \Bigg( \sum_{i=1}^{n} \frac{\partial a_i}{\partial x_i} \Bigg) \omega_X
\end{gather*}
holds, where $ \omega_X={\rm d}x_1 \wedge {\rm d}x_2 \wedge \cdots \wedge {\rm d}x_n$.
\end{Proposition}
\begin{proof}
Let $ \beta=i_v(\omega_X)$.
Since $ {\rm d}f\wedge \beta = v(f) \omega_X, $ we have $ {\rm d}f\wedge \beta = \left( \sum_{i=1}^{n} a_i(x)\frac{\partial f}{\partial x_i} \right)\omega_X$. Since $ v(f) =b(x)f(x), $ Brieskorn formula implies the result.
\end{proof}


Notice that the action of $Df $ on $ b(x)\omega_X $ in the formula above is completely written in terms of non-trivial logarithmic vector field
$v$ such that $ v(f)=b(x)f$.
To the best of our knowledge, this simple observation has not been explicitly stated in literature on Gauss--Manin connections.


Now we present an alternative method for computing the module of germs of non-trivial logarithmic vector fields.
\begin{enumerate}\itemsep=0pt
\setlength{\leftskip}{0.45cm}
\item[{\it Step} 1:] Compute a monomial basis ${\rm M} $ of the quotient space
\begin{gather*}
{\mathcal O}_{X,O}\Big/\bigg(\frac{\partial f}{\partial x_1}, \frac{\partial f}{\partial x_2}, \dots, \frac{\partial f}{\partial x_n}\bigg).
\end{gather*}

\item[{\it Step} 2:] Compute a standard basis ${\rm Sb}$ of the ideal quotient
\begin{gather*}
\bigg(\frac{\partial f}{\partial x_1}, \frac{\partial f}{\partial x_2}, \dots, \frac{\partial f}{\partial x_n}\bigg) : (f).
\end{gather*}
\item[{\it Step} 3:] Compute a basis $ {\rm B} $ of the vector space by using $ {\rm Sb} $ and $ {\rm M} $
\begin{gather*}
\bigg(\bigg(\frac{\partial f}{\partial x_1}, \frac{\partial f}{\partial x_2}, \dots, \frac{\partial f}{\partial x_n}\bigg) : (f)\bigg)\Big/
\bigg(\frac{\partial f}{\partial x_1}, \frac{\partial f}{\partial x_2}, \dots, \frac{\partial f}{\partial x_n}\bigg).
\end{gather*}
\item[{\it Step} 4:] For each $ b(x) \in {\rm B}, $ compute a logarithmic vector field along $S$ such that
\begin{gather*}
v(f)=b(x)f(x).
\end{gather*}
\end{enumerate}

The method above computes a basis of non-trivial logarithmic vector fields.
Each step can be effectively executable, as in~\cite{TN20}, by utilizing algorithms described in~\cite{NT16a,NT16b,NT17a,TNN}.

Note that, the number of non-trivial logarithmic vector fields in the output is equals to the Tjurina number $ \tau(f)$. See also~\cite{M}.

 Let
\begin{gather*}
v=a_1(x)\frac{\partial}{\partial x_1} + a_2(x)\frac{\partial}{\partial x_2}+ \cdots +a_n(x)\frac{\partial}{\partial x_n}
\end{gather*}
be a germ of non-trivial logarithmic vector field along $S, $ such that $ v(f) = b(x)f(x)$. Then from Proposition~\ref{pro5},
we have
\begin{gather*}
D(f(x)b(x)\omega_X) = \left( \sum_{i=1}^{n} \frac{\partial a_i}{\partial x_i} \right) \omega_X.
\end{gather*}
Therefore, the proposed method can be used as a basic procedure for computing a connection matrix of Gauss--Manin connection.

One of the advantages of the proposed method lies in the fact that the resulting algorithm also can handle parametric cases.

\subsection{Examples}

Let us recall that $ x^3+y^7+txy^5 $ is the standard normal form of semi quasi-homogeneous $E_{12}$ singularity. The weight vector is $(7,3)$ and the weighted degree of the quasi-homogeneous part is equal to $21$ and the weighted degree of the upper monomial $ txy^5 $ is equal to $22$. We examine here, by contrast, the case where the weighted degree of an upper monomial is bigger than $ 22.$

\begin{Example}
Let $ f(x,y)=x^3+y^7+txy^6, $ where $t$ is a parameter. Notice that the polyno\-mial~$f$ is not weighted homogeneous. The weighted degree of the upper monomial $ txy^6 $ is equal to $25$, which is bigger than that of $ txy^5$. Accordingly $f$ is a quasi homogeneous function. The Milnor number $\mu$ is equal to~$12$.

Let $H_J $ denote the set of local cohomology classes in $ H_{[0,0]}^{2}({\mathcal O}_{X}) $ that are killed by the Jacobi ideal $J=\big (\frac{\partial f}{\partial x}, \frac{\partial f}{\partial y}\big)$:
\begin{gather*}
H_J = \left\{ \psi \in H_{[0,0]}^{2}({\mathcal O}_{X})\, \middle| \,\frac{\partial f}{\partial x}\psi = \frac{\partial f}{\partial y}\psi = 0 \right\}.
\end{gather*}
Then, by using an algorithm given in~\cite{NT17a,TNN}, a basis as a vector space of $H_J$ is computed as
\begin{gather*}
\begin{bmatrix} 1 \\ x y \end{bmatrix},
\begin{bmatrix} 1 \\ x y^2 \end{bmatrix},
\begin{bmatrix} 1 \\ x y^3 \end{bmatrix},
\begin{bmatrix} 1 \\ x^2 y \end{bmatrix},
\begin{bmatrix} 1 \\ x y^4 \end{bmatrix},
\begin{bmatrix} 1 \\ x^2 y^2 \end{bmatrix},
\begin{bmatrix} 1 \\ x y^5 \end{bmatrix},
\begin{bmatrix} 1 \\ x^2 y^3 \end{bmatrix},
\begin{bmatrix} 1 \\ x y^6 \end{bmatrix},
\begin{bmatrix} 1 \\ x ^2 y^4 \end{bmatrix},
\begin{bmatrix} 1 \\ x^2 y^5 \end{bmatrix},
\\
\begin{bmatrix} 1 \\ x^2 y^6 \end{bmatrix} -\dfrac{6}{7}t
\begin{bmatrix} 1 \\ x y^7 \end{bmatrix}+\dfrac{2}{7}t^2
\begin{bmatrix} 1 \\ x^3 y \end{bmatrix},
\end{gather*}
where $[\ ] $ stands for Grothendieck symbol.

It is easy to see that every local cohomology classes in $H_J$ is killed by $f$, that is $ f\cdot \varphi =0$, $\forall\, \varphi \in H_J$. Therefore, $f$ is in the ideal $\big(\frac{\partial f}{\partial x}, \frac{\partial f}{\partial y}\big) \subset {\mathcal O}_{X,O}$.

Therefore, by a classical result of K.~Saito~\cite{S71}, $f$ is in fact quasi-homogeneous. The Tjurina number $\tau$ is equal to the Milnor number $\mu=12$.
A monomial basis ${\rm M}$ of $ {\mathcal O}_{X,O}/\big(\frac{\partial f}{\partial x}, \frac{\partial f}{\partial y}\big) $ is
\begin{gather*}
{\rm M} = \big\{1, y, y^2, x, y^3, xy, y^4, xy^2, y^5, xy^3, xy^4, xy^5 \big\}.
\end{gather*}
Since a standard basis ${\rm Sb} $ of $\big(\frac{\partial f}{\partial x}, \frac{\partial f}{\partial y}\big) : (f)$ is $\{1\},$ a basis ${\rm B}$ of the vector space
$ \big(\big(\frac{\partial f}{\partial x}, \frac{\partial f}{\partial y}\big) : (f)\big)/ \big(\frac{\partial f}{\partial x}, \frac{\partial f}{\partial y}\big)$ is
equal to ${\rm M}$ that consists of $\tau=12$ elements.

By using an algorithm given in~\cite{NT16b}, we compute a logarithmic vector field which plays the role of Euler vector field.
The result of computation is the following:
\begin{gather*}
v= \frac{d_1}{3\big(49+12t^3y^4\big)}\frac{\partial}{\partial x} + \frac{d_2}{3\big(49+12t^3y^4\big)}\frac{\partial}{\partial y},
\end{gather*}
where
\begin{gather*}
d_1=49x+8t^2y^5+12t^3xy^4, \qquad d_2=21y-4tx+4t^3y^5.
\end{gather*}
The vector field $ v$ enjoys $ v(f)=f$. Note also that for the case $ t=0, $ we have
\begin{gather*}
v= \frac{1}{21}\bigg(7x\frac{\partial}{\partial x} + 3y\frac{\partial}{\partial y}\bigg).
\end{gather*}

We emphasize here the fact that, the algorithm in~\cite{NT19a} for computing logarithmic vector fields can handle parametric cases.
Since $v(f)=f$ holds, the other non-trivial logarithmic vector fields can be obtained from~$v$.
In fact, for $x^i y^j \in M$, we have $x^i y^j v(f) = x^i y^j f$.

Therefore, thanks to Brieskorn formula, Gauss--Manin connection can be determined expli\-ci\-tly by using these non-trivial logarithmic vector fields,
\end{Example}

\begin{Remark}Recall that, according to Grothendieck local duality theorem, the vector space~$H_J$ can be regarded as a dual space to ${\mathcal O}_{X,O}/J$. Since these local cohomology classes given above constitute a~dual basis of the monomial basis~M of the quotient space $ {\mathcal O}_{X,O}/J$, the normal form of a holomorphic function w.r.t.~${\mathcal O}_{X,O}/J$ can be computed by using the basis of~$H_J$
in an efficient manner, without using division algorithms~\cite{TNN}.

Therefore the use of local cohomology classes in reduction steps allows us to design an effective procedure for computing the connection matrix of Gauss--Manin connection.
\end{Remark}

J.~Scherk studied in~\cite{Sch} the following case.

\begin{Example}
Let $f(x,y)=x^5+x^2y^2+y^5$. Then, the Milnor number $\mu(f)$ is equal to~11 and the Tjurina number $\tau(f)$ is equal to 10. A monomial basis ${\rm M}$ of ${\mathcal O}_{X,O}/\big(\frac{\partial f}{\partial x}, \frac{\partial f}{\partial y}\big)$ is~${\rm M} = \big\{1, x, x^2, x^3, x^4, x^5, xy, y, y^2, y^3, y^4\big\}$. A standard basis ${\rm Sb}$ of the ideal quotient
$\big(\frac{\partial f}{\partial x} \frac{\partial f}{\partial y}\big) : (f)$ is $\{x, y\}$. A basis ${\rm B}$ of the vector space $\big(\big(\frac{\partial f}{\partial x}, \frac{\partial f}{\partial y}\big) : (f)\big) / \big(\frac{\partial f}{\partial x}, \frac{\partial f}{\partial y}\big)$ is
\begin{gather*}
{\rm B} = \big\{ x, x^2, x^3, x^4, x^5, xy, y, y^2, y^3, y^4 \big\}.
\end{gather*}

Since $ {\rm Sb}\cap {\rm B} = \{ x, y \} $, we first compute non-trivial logarithmic vector fields associated to $ x $ and $y $.
\begin{enumerate}\itemsep=0pt
\item[$(i)$] For $b(x,y)=x$, we have
\begin{gather*}
v=\frac{d_1}{5(4-25xy)}\frac{\partial}{\partial x} + \frac{d_2}{5(4-25xy)}\frac{\partial}{\partial y},
\end{gather*}
where
$d_1=4x^2-25x^3y-5y^3$, $d_2=6xy-25x^2y^2$.

Since $v(f)=xf$, by a direct computation, we have for instance
\begin{gather*}
D(f(x, y)x\omega_X) = \left(\dfrac{7}{10}x-\dfrac{3\times25}{16}y^4\right)\omega_X \mod \left(\frac{\partial f}{\partial x}, \frac{\partial f}{\partial y}\right) .
\end{gather*}

Since $x^iv(f) =x^{i+1}f$, $i=1,2,3,4$ and $yv(f)=xyf$ hold, we can compute the action of~$Df$ on $x^{i+1}\omega_X$ and $xy\omega_X$
by using the vector field $v $ above.

\item[$(ii)$] For $b(x,y)=y$, we have
\begin{gather*}
v=\frac{d_1}{5(4-25xy)}\frac{\partial}{\partial x} + \frac{d_2}{5(4-25xy)}\frac{\partial}{\partial y},
\end{gather*}
where
$d_1=6xy-25x^2y^2$, $d_2=4y^2-25xy^3-5x^3$ and
\begin{gather*}
D(f(x,y)y\omega_X) = \left(\dfrac{7}{10}y-\dfrac{3\times25}{16}x^4\right)\omega_X \mod \left(\frac{\partial f}{\partial x}, \frac{\partial f}{\partial y}\right) .
\end{gather*}
\end{enumerate}

Since the vector field $v$ above satisfies $v(f) = yf$, we also have $y^jv(f) = y^{j+1}f$, $j=1,2,3$.

We can use these relations to compute the action of $Df$ on $y^{j+1} \omega_X$, $j=1,2,3$.
In this way, we obtain $\tau=10$ fundamental relations.

Since the Milnor number $\mu$ is equal to 11, these 10 relations are not enough to compute a~connection matrix of the Gauss--Manin connection.
We have to compute the saturation.

Now recall the classical result on integral closure due to J.~Brian\c on and H.~Skoda~\cite{BSkoda}.
From the Brian\c on--Skoda theorem, we see that the function $f^2$ is in the ideal
$J=\big(\frac{\partial f}{\partial x}, \frac{\partial f}{\partial y}\big)$.
In~\cite{Sch}, J.~Scherk computed the following relation explicitly and exploited it as the starting point for computing $D(f^2\omega_X)$ and $D(fD(f\omega_X))$:
\begin{gather*}
25(4-25xy)f^2 = \big\{\big(20x-125x^2y\big)f+4x^3y^2-5xy^5-25x^4y^3\big\}\frac{\partial f}{\partial x}
\\ \hphantom{25(4-25xy)f^2 = \{}{}
+ \big\{\big(20y-125xy^2\big)f+6x^2y^3-25x^3y^4\big\}\frac{\partial f}{\partial y}.
\end{gather*}

Here we propose a slightly different approach.
By using an algorithm given in~\cite{NT20}, we can compute the following integral dependence relation
\begin{gather*}
25(4-25xy)f^2 = 10x\bigg(\frac{\partial f}{\partial x}\bigg)f\!+\!10y\bigg(\frac{\partial f}{\partial y}\bigg)f\!+\!d_{2,0}\bigg(\frac{\partial f}{\partial x}\bigg)^2\!+\!d_{1,1}\bigg(\frac{\partial f}{\partial x}\bigg)\bigg(\frac{\partial f}{\partial y}\bigg)\!+\!d_{0,2}\bigg(\frac{\partial f}{\partial y}\bigg)^2,
\end{gather*}
where
\begin{gather*}
d_{2,0}=2x^2-25x^3y-10y^3, \qquad
d_{1,1}=11xy-50x^2y^2, \qquad
d_{0,2}=2y^2-25xy^3-10x^3.
\end{gather*}

Compare to the relation used by Scherk, the integral dependence relation given above represent much more precise relations between $f^2$, $f\big(\frac{\partial f}{\partial x}\big)$,
$f\big(\frac{\partial f}{\partial y}\big)$, $\big(\frac{\partial f}{\partial x}\big)^2$,
$\big(\frac{\partial f}{\partial x}\big)\big(\frac{\partial f}{\partial y}\big)$,
$\big(\frac{\partial f}{\partial y}\big)^2$. Thanks to this property,
the use of the integral dependence relation, or the integral equation leads an~effective method for computing $ D\big(f^2\omega_X\big) $ and $ D(fD(f\omega_X))$.
\end{Example}

Note that in~\cite{NT20}, we consider integral dependence relations in the context of symbolic computation and introduced a concept of generalized integral dependence relations. From this point of~view relations obtained from non-trivial logarithmic vector fields can be interpreted as genera\-lized integral dependence relations. These relations can also be computed by using the algorithms described in~\cite{NT20}.

Let $ f(x) $ be a holomorphic function defined on $X \subset {\mathbb C}^n$. Assume that the degree of integral equation, or the integral number of $f$ over the Jacobi ideal in the local ring ${\mathcal O}_{X,O}$ is equal to~two. Let
\[
f(x)^2 + \sum_{i=1}^{n} a_i(x)f(x)\frac{\partial f}{\partial x_i}(x)+ \sum_{j \ge i} a_{i,j}(x)
\frac{\partial f}{\partial x_i}(x)\frac{\partial f}{\partial x_j}(x) = 0
\]
be the integral equation of $f$. Then, from the Brieskorn formula, we have
\begin{align*}
D(f(x)^2\omega_X) &= -D\Bigg\{ \sum_{i=1}^{n}\Bigg(a_i(x)f(x) + \sum_{j \ge i} a_{i,j}(x)\frac{\partial f}{\partial x_j}(x)\Bigg)\frac{\partial f}{\partial x_i}(x) \omega_X \Bigg\}
\\
&=-\Bigg\{ \sum_{i=1}^{n}\frac{\partial}{\partial x_i} \Bigg(a_i(x)f(x) + \sum_{j \ge i}
 a_{i,j}(x)\frac{\partial f}{\partial x_j}(x)\Bigg) \Bigg\} \omega_X
\end{align*}
which is equal to
\[
-\Bigg\{ \sum_{i=1}^{n}a_i(x)\frac{\partial f}{\partial x_i}(x) +
 \sum_{j \ge i} \frac{\partial a_{i,j}}{\partial x_i}(x)\frac{\partial f} {\partial x_j} + R(x) \Bigg\} \omega_X,
 \]
where
\[
R(x) = \bigg( \sum_{i=1}^{n} \frac{\partial a_i}{\partial x_i}(x)\bigg)f(x) +
 \sum_{j \ge i} a_{i,j}(x)\frac{\partial^2 f}{\partial x_i \partial x_j} (x).
 \]

If, there exist holomorphic functions $c_i(x)$, $i=1,2,\ldots,n$ such that
\[
R(x) = \sum_{i=1}^{n} c_i(x)\frac{\partial f}{\partial x_i}(x),
\]
then, we have for instance the following relation that can be used as a starting point of the computation of a saturation
\[
D^2(f(x)^2\omega_X) = -\Bigg\{ \sum_{i=1}^{n} \bigg(\frac{\partial a_i}{\partial x_i}(x)
+ \frac{\partial c_i}{\partial x_i}(x) \bigg) + \sum_{ j \ge i}
\frac{\partial^2 a_{i,j}}{\partial x_j \partial x_i}(x) \Bigg\} \omega_X.
 \]

Computing Gauss--Manin connections is a quite difficult problem~\cite{D,Gu,Sch,Sch93,vS}. We~expect that the approach presented in this paper provides a method to reduce difficulty to some extent.

\subsection*{Acknowledgements}
{\sloppy
This work has been partly supported by JSPS Grant-in-Aid for Scientific Research (C) (18K03320 and 18K03214).

}

\pdfbookmark[1]{References}{ref}
\LastPageEnding

\end{document}